\documentclass[12pt,twoside]{amsart}
\usepackage[latin1]{inputenc}
\usepackage{stmaryrd}
\usepackage{times}
\usepackage{amsthm}
\usepackage{amssymb, verbatim}
\usepackage{amsmath}
\usepackage{mathrsfs}

\topmargin = - 40 pt %distanza dal margine superiore 0= un bel po` negativo diminuisce
\newcommand     {\Rset}    {{\mathbb R}}

\newcommand{\SL}{{\mathrm {SL}}}

\newcommand     {\mC}    {{\mathbb C}}
\newcommand     {\mR}    {{\mathbb R}}
\newcommand     {\gog}    {{\frak g}}
\newcommand     {\got}    {{\frak t}}
\newcommand     {\gos}    {{\frak s}}
\newcommand     {\gou}    {{\frak u}}
\newcommand     {\goh}    {{\frak h}}
\newcommand     {\gom}    {{\frak m}}
\newcommand     {\gra}    {{\alpha}}
\newcommand     {\grb}    {{\beta}}
\newcommand     {\lra}   {{\longrightarrow}}

\newcommand     {\G}    {{\frak g}}

\newcommand     {\id}   {\mathrm{id}}

\newtheorem{thm}{Theorem}[section]

\newtheorem{lemma}{Lemma}[section]
\newtheorem{prop}{Proposition}[section]
\newtheorem{cor}{Corollary}[section]

\theoremstyle{definition}

\DeclareMathOperator{\ad}{ad}
\DeclareMathOperator{\Ad}{Ad}

\DeclareMathOperator{\comm}{comm}
\DeclareMathOperator{\Comm}{Comm}

\numberwithin{equation}{section}

\begin{document}

\title[Small commutators in compact semisimple Lie groups and Lie algebras]{Small commutators in compact semisimple\\ Lie groups and Lie algebras}
\author[A.~D'Andrea]{Alessandro D'Andrea}
\author[A.~Maffei]{Andrea Maffei}
%\thanks{The author was partially supported by
%PRIN ``Spazi di Moduli e Teoria di Lie'' fundings from MIUR and
%project MRTN-CT 2003-505078 ``LieGrits'' of the European Union}
%\address{Dipartimento di Matematica, Universit\`a degli Studi di
%Roma ``La Sapienza'', Roma}
%\email{dandrea@mat.uniroma1.it}

\date{\today}

\maketitle

\tableofcontents

\section{Introduction}

%\subsection{What is known?}

Let $G$ be a connected semisimple (Lie or algebraic) group. Then $G$ equals its derived subgroup and it is expected that, in many cases, every element of $G$ is indeed a commutator. The problem of understanding under what conditions this claim holds, or at least every element can be expressed as a product of a uniformly bounded quantity of commutators, has been investigated at length.

The fact that every element in a semisimple compact Lie group is a commutator dates back to Got\^o \cite{G}, whereas counterexamples are easy to construct in non compact cases --- for instance, in $\SL_2(\Rset)$, $-\id$ does not arise as a commutator. Later, Thompson \cite{T} provided a classification of all groups of the form $\SL_n(k),$ where $k$ is an arbitrary field, containing elements that are not commutators.

Connected semisimple groups are treated in the complex case in \cite{PW}, and in \cite{R} over an algebraically closed field of any characteristic. More recently, \DH okovi\'c showed \cite{D}, under mild technical assumptions, that in the real semisimple case every element is a product of at most two commutators.

Many variations on the topic have also been considered. To name just a few, Brown considered the analogous statement in the case of simple Lie algebras \cite{Brown}; Borel studied instead maps $G^n \to G$ induced by nontrivial group words in $n$ letters, showing \cite{Borel} that they yield dominant maps.

In this paper, we establish a different property of the commutator map in semisimple compact Lie groups and Lie algebras: its openness at the identity element. We say that a map $f$ between two topological spaces $X$ and $Y$ is open at a point $x\in X$ if for all neighborhoods $U$ of $x$ the image $f(U)$ is a neighborhood of $f(x)$, so that our claim amounts to showing that elements that are sufficiently small (i.e., close to the identity element) arise as commutators of prescribedly small elements. The usual proofs that in a compact semisimple Lie group every element is a commutator provide little information towards this statement as they proceed by expressing each element in the group as the commutator between an element lying in a torus and some expression, which is typically ``far'' from the identity, depending on a nontrivial Coxeter element chosen in the associated Weyl group.

In Section \ref{Liealgebras}, we treat the infinitesimal case of compact semisimple Lie algebras: in this setting, the commutator map is a surjective bilinear map. It was a classical problem answered negatively by Cohen and Horowitz to establish whether the surjectivity of a bilinear map implies its openess in zero \cite{C, H%, Dixon
}.

We show that the commutator map of compact semisimple Lie algebras can be inductively proved to be open by exploiting combinatorial properties of the corresponding root systems. The basis of induction corresponds to Lie algebras of type $A$, and needs to be done by hand.

%\subsection{What is our strategy?}

Our next step is to {\em integrate} to the group level the knowledge we have gathered for Lie algebras. Once more, this is not totally immediate: indeed, the commutator map for a Lie algebra is only a second order approximation of the commutator map for the corresponding Lie group, and we face a scarcity of tools for translating second order information from the infinitesimal level to the local one.

In Section \ref{Liegroups}, we solve this technical issue in two steps. First, we integrate half of the commutator map of the Lie algebra $\G$ to the map
$$\G \times \G \ni (x, y) \mapsto x - \exp(\ad_y) x \in \G,$$
and then use techniques from Rouvi{\`e}re \cite{Rou}, related to the Kashiwara-Vergne method, to get to the group level.

We should stress that our strategy employs more than once the fact that all elements in a compact Lie group (resp. Lie algebra) lie in a torus, and therefore does not immediately extend to noncompact structures.

We would like to thank Alessandro Berarducci for drawing our attention to this problem, which arises from  his work on definable groups. Our openness statement is equivalent to the claim that every element belonging to the infinitesimal
neighbourhood of the identity  (which is a perfect subgroup) in the
non-standard version of a compact semisimple Lie group is a
commutator. This issue was considered by Berarducci, Peterzil and Pillay --- see in particular the comments after \cite[Proposition 2.14]{APP} --- in connection to the question whether a finite central extension of a group definable in an $o$-minimal structure $M$ is interpretable in $M$. A positive answer to the latter question would also imply that a finite central extension of a compact Lie group has an induced Lie structure making the extension a topological cover.

\section{Openness of the commutator map in a compact semisimple Lie algebra}\label{Liealgebras}

Throughout the paper, $G$ will be a semisimple Lie group and $\gog$ its Lie algebra. 
On $\gog$ we consider the Killing form $\kappa_\gog$.

In the following Lemma we characterize pairs of maximal toral subalgebras of
$\gos\gou_n$ which are orthogonal to each other with respect to the Killing form.
If $u_1,\dots,u_n$ is an orthonormal basis of $\mC^n$ then we denote by $\got_u$
the set of elements in $\gos\gou_n$ which are diagonal with respect to
this basis.

\begin{lemma}\label{lem:toriortogonali}
Let $u_1,\dots,u_n$ and $v_1,\dots,v_n$ be two orthonormal bases of
$\mC^n$. Then $\got_u$ is orthogonal to $\got_v$ if and only if 
\begin{equation}\label{ortho}
|u_i\cdot v_h| =  |u_j\cdot v_k|,
\end{equation}
for all $i,j,h,k$.
\end{lemma}
\begin{proof}
Define $U_{ij} \in \gos\gou_n$ by
$$U_{ij}(u_h) =
\begin{cases} 
\sqrt{-1}u_i &\text{if }h=i;\\ 
-\sqrt{-1}u_j &\text{if }h= j;\\
0 &\text{otherwise}.
\end{cases}$$
and similarly define $V_{ij}$ using the orthonormal basis $v_i$.
Then the operators $U_{ij}$ span $\got_u$ and the operators $V_{ij}$ span $\got_v$.
Easy computations show that
$$
\text{Tr } \big(U_{ij}V_{hk}\big)= |u_i\cdot v_k|^2+|u_j\cdot v_h|^2-|u_i\cdot v_h|^2-|u_j\cdot v_k|^2.
$$
Hence $\got_u$ is orthogonal to $\got_v$ if and only if 
$$
|u_i\cdot v_k|^2+|u_j\cdot v_h|^2 = |u_i\cdot v_h|^2 + |u_j\cdot v_k|^2
$$
for all $i,j,h,k$. Hence if Equation \ref{ortho} is verified, then the two subalgebras are orthogonal. Vice versa,
assume they are orthogonal. Then, summing over $h$, we see that the above equalities imply
$$
|u_i\cdot v_k|^2= |u_j\cdot v_k|^2
$$
for all $i,j$ and summing over $j$ we get 
$$
|u_i\cdot v_h|^2= |u_i\cdot v_k|^2
$$
proving the claim.
\end{proof}

\begin{lemma}
Let $G$ be compact, $\got$ be a maximal toral subalgebra of $\gog$. Then there exists a maximal toral subalgebra of $\gog$ orthogonal to $\got$.
\end{lemma}
\begin{proof}
We first analyse the case $\gog=\gos\gou_n$. Set $\zeta = e^{\frac{2\pi i}{n}}$, and
let $u_1,\dots,u_n$ be an orthonormal basis of $\mC^n$ such that
$\got=\got_u$. For $j=1,\dots,n$ define
$$
v_j = \frac{1}{\sqrt{n}} (u_1+\zeta^j u_2+ \zeta^{2j}u_2+\dots+\zeta^{(n-1)j}u_n).
$$
Then $v_1,\dots,v_n$ is an orthonormal basis of $\mC^n$ and $|u_i\cdot v_j|=1/\sqrt{n}$ for all $i,j$. Hence $\got_u$ is orthogonal to $\got_v$.
For $\gog$ not isomorphic to $\gos\gou_n$ we proceed by induction on
the rank of $\gog$. If $\gog$ is not simple the claim follows
immediately by induction, so we assume that $\gog$ is simple.

Let $\gog_\mC$ (resp. $\got_\mC$) be the complexification of $\gog$
(resp. $\got_\mC$), denote by $\Phi$ the associated root system and
choose a simple basis $\Delta\subset \Phi$.  Let $\omega_\alpha$, for $\alpha \in \Delta,$ be the corresponding fundamental weights, and $\theta$ be the highest root of $\Phi$.
Since $\gog$ is not of type $A$ there exists a simple root $\alpha$ such that $\theta =\omega_\alpha $ or $\theta=2\omega_\alpha$, see \cite[Planches II-IX]{Bou}. Let $\Psi$ be the root
system generated by $\Delta\setminus \{\gra\}$.

We can choose a standard Chevalley basis, $h_\alpha$ with $\gra \in \Delta$ and $x_\gra$ with $\gra \in \Phi$ such that elements 
$$
k_\gra = \sqrt{-1} h_\gra \qquad u_\gra= x_\gra-x_{-\gra} \;\text{ and }\;v_\gra = \sqrt{-1}(x_\gra+x_{-\gra})
$$
are a basis of $\gog$, see \cite[Theorem 6.11, Formula (6.12)]{Kna}. Notice that the subspace orthogonal of $\got$ is the linear span of the elements $u_\gra$ and $v_\gra$.
Define $$ \goh =\langle k_\grb, u_\grb, v_\grb : \grb \in \Psi \rangle.$$ 
This is the semisimple part of the maximal Levi subalgebra associated to $\gra$.  In particular the claim is true for $\goh$. Let $\gos$ be a maximal toral subalgebra orthogonal to the maximal toral subalgebra of $\goh$ given by $\got\cap \goh$.

Notice also that we have $[u_\theta,\goh]=[u_{-\theta},\goh]=0$. Hence, for dimension reasons,
$$
\gos\oplus \mR u_\theta
$$
is a maximal toral subalgebra of $\gog$ orthogonal to $\got$.
\end{proof}

We can now prove the following fact.
\begin{thm}
Let $G$ be compact. Then the commutator map
$$\comm_\G: \G \times \G \ni (x, y) \mapsto [x, y] \in \G$$
is open at $(0, 0)$.
\end{thm}
\begin{proof}
We need to prove that if $U$ is neighbourhood of $0$ then $\comm_\G(U\times U)$ contains a neighbourhood of $0$. Notice first that being $G$ compact we can assume that $U$ is $G$-stable under the adjoint action.

Choose now a regular element $x \in U$ (an element is said to be regular if its centralizer is a toral subalgebra) and let $\got$ be its centralizer. Let $\gom$ be the orthogonal of $\got$. Then 
$$
\ad_x:\gom \lra \gom
$$
is a linear isomorphism. Hence there exists a $G$-stable neighbourhood $V$ of $G$ such that 
$\ad_x(\gom\cap U)\supset \gom\cap V$.

Consider now $\psi:G \times (\gom\cap U)$ given by
$\psi(g,y)=\Ad_g [x,y]$. It is clear that the image of $\psi$ is
contained in $\comm_\G( U\times U)$ and that the image of $\psi$
contains $G\cdot(V\cap \gom)$.
Finally from the previous Lemma we have that $G\cdot (V\cap\gom) = V$. Hence 
$\comm_\G( U\times U) \supset V$ proving the Theorem.
\end{proof}

\section{Openness of the commutator map in a Lie group}\label{Liegroups}

We will now show that the commutator map
$$\Comm_G: G \times G \ni (X, Y) \mapsto XYX^{-1}Y^{-1} \in G$$
is open at $(\id, \id)$ as soon as the corresponding infinitesimal commutator map $\comm_\G: \G \times \G \to \G$ is open at $(0,0)$.

\begin{lemma}
The map $\comm_\G$ is open at $(0,0)$ if and only if the map
$$C_\G: \G \times \G \ni (x, y) \mapsto x - \exp(\ad_y) x \in \G$$
is so.
\end{lemma}
\begin{proof}
The map
$$\phi: \G \times \G \ni (x, y) \mapsto \left( \frac{\exp(\ad_y) - 1}{\ad_y}(x), y \right) \in \G \times \G$$
is smooth, and has invertible differential in $(0,0)$, so it is a local diffeomorphism. However, the composition of $\comm_\G \circ \, \phi$ equals $C_\G$.
\end{proof}
In order to deal with openness of the commutator map in a group, we are going to use the following variant \cite[Remarque 4.2]{Rou}, related to the Kashiwara-Vergne method, of the Baker-Campbell-Hausdorff formula.
\begin{prop}\label{KW}
There exist, in a neighbourhood of $(0, 0) \in \G \times \G$, analytical functions
$$P, Q: \G \times \G \to G, \qquad P(0, 0) = Q(0, 0) = \id,$$
satisfying
$$\exp(a + b) = \exp(P(a, b).a) \exp(Q(a, b).b),$$
for all $a, b$.
\end{prop}
\begin{thm}
The group commutator map $\Comm_G$ is open at $(\id, \id)$ as soon as the infinitesimal commutator map $\comm_\G$ is open at $(0, 0)$.
\end{thm}
\begin{proof}
Let us apply Proposition \ref{KW} to $a = x, b = - \exp(\ad_y) x$. Using the notation introduced above, we set $P = P(a, b), Q = Q(a, b)$ and obtain
\begin{eqnarray*}
\exp(x - \exp(\ad_y) x) & = & \exp(P.x) \exp(-Q.\exp(\ad_y) x)\\
& = & P \exp(x) P^{-1} (Q\exp(y)) \exp(-x) (Q \exp(y))^{-1}\\
& = & ABA^{-1}B^{-1},
\end{eqnarray*}
where $A = P \exp(x) P^{-1}, B = Q \exp(y) P^{-1}$.

Let $U$ be a ($G$-stable) neighbourhood of $0\in \G$ on which $\exp$ restricts to a diffeomorphism. The map $\psi: \G \times \G \to G$ defined by
$$(x, y) \mapsto Q(x, \exp(\ad_y)x) \exp(y) P(x, \exp(\ad_y)x)^{-1}$$
is analytical, hence continuous. We may then find a neighbourhood $U'$ of $0 \in \G$, which we assume to be $G$-stable and contained in $U$, such that $\psi(U' \times U') \subset \exp U$. If $x, y$ lie in $U'$, then $A$ and $B = \psi(x, y)$ lie in $\exp U$; moreover, the composition $\exp \circ \,C_\G$ maps $(x, y)$ in $ABA^{-1}B^{-1}$ and is open at $(0, 0)$.

We thus conclude that all elements in a suitable neighbourhood of $\id \in G$ arise as commutators of elements from $\exp U$.
\end{proof}
\begin{cor}
If $G$ is a compact semisimple Lie group, then $\Comm_G$ is open at $(\id, \id)$.
\end{cor}


\begin{thebibliography}{99.}
\bibitem{APP}
Alessandro Berarducci, Ya'acov Peterzil and Anand Pillay,
\textit{Group covers, $o$-minimality, and categoricity},
Confluentes Math. \textbf{2}, no. 4 (2010), 473-496.

\bibitem{Borel}
Armand Borel,
\textit{On free subgroups of semisimple groups}.
Enseign. Math. (2) \textbf{29}, no. 1-2 (1983), 151--164.

\bibitem{Bou}
Nicholas Bourbaki,
\textit{Groupes et alg\`ebres des Lie. Chapitres 4, 5 et 6}. \'El\'ements de Math\'ematique, Hermann 1968, Masson 1981.

\bibitem{Brown}
Gordon Brown,
\textit{On commutators in a simple Lie algebra},
Proc. Amer. Math. Soc. \textbf{14} (1963), 763--767.

\bibitem{C}
Paul J. Cohen,
\textit{A counterexample to the closed graph theorem for bilinear maps},
J. Funct. Anal.  \textbf{16} (1974), 235--239. 

\bibitem{D}
Dragomir \v{Z}. \DH okovi\'c,
\textit{On commutators in real semisimple Lie groups},
Osaka J. Math. \textbf{23} (1986), 223--228.

\bibitem{G}
Morikuni Gotô,
\textit{A theorem on compact semi-simple groups},
J. Math. Soc. Japan \textbf{1} (1949), 270--272.

\bibitem{H}
Charles Horowitz,
\textit{An elementary counterexample to the open mapping principle for bilinear maps}, 
Proc. Amer. Math. Soc. \textbf{53} (1975), 293--294. 

\bibitem{Kna}
Anthony W. Knapp,
\textit{Lie groups beyond an introduction, 2nd edition}. Progress in Mathematics \textbf{140}, Birkh\"auser, 2002.

\bibitem{PW}
Samuel Pasiencier and Hsien-Chung Wang,
\textit{Commutators in a semi-simple group},
Proc. Amer. Math. Soc. \textbf{13} (1962), 907--913.

\bibitem{R}
Rimhak Ree,
\textit{Commutators in semi-simple algebraic groups},
Proc. Amer. Math. Soc. \textbf{15} (1964), 457--460.

\bibitem{Rou}
Fran\c{c}ois Rouvi{\`e}re,
\textit{Espaces sym\'etriques et m\'ethode de Kashiwara-Vergne},
Ann. Sci. \'Ecole Norm. Sup. \textbf{19}, no. 4 (1986), 553--581.

\bibitem{T}
Robert C. Thompson,
\textit{Commutators in the special and general linear groups},
Trans. Amer. Math. Soc. \textbf{101} (1961), 16--33.

\end{thebibliography}
\end{document}